\documentclass[pdflatex,sn-mathphys-num]{sn-jnl}

\usepackage{graphicx}%
\usepackage{multirow}%
\usepackage{amsmath,amssymb,amsfonts}%
\usepackage{amsthm}%
\usepackage{mathrsfs}%
\usepackage[title]{appendix}%
\usepackage{xcolor}%
\usepackage{textcomp}%
\usepackage{manyfoot}%
\usepackage{booktabs}%
\usepackage{algorithm}%
\usepackage{algorithmicx}%
\usepackage{algpseudocode}%
\usepackage{listings}%

\theoremstyle{thmstyleone}%
\newtheorem{theorem}{Theorem}
\newtheorem{proposition}[theorem]{Proposition}%
\newtheorem{fact}{Fact}
\newtheorem{lemma}{Lemma}
\newtheorem{corollary}{Corollary}

\theoremstyle{thmstyletwo}%
\newtheorem{example}{Example}%
\newtheorem{remark}{Remark}%

\theoremstyle{thmstylethree}%
\newtheorem{definition}{Definition}%

\newtheoremstyle{myextra}
  {3pt}      
  {3pt}      
  {\itshape} 
  {}         
  {\bfseries}
  {.}        
  {1em}      
  {}         

\theoremstyle{myextra}%
\newtheorem{examplecustom}{Example}%
\newtheorem{remarkcustom}{Remark}%
  
\newcommand{\sgn}{\operatorname{sgn}}

\begin{document}

\title{Generating Sets of Stochastic Matrices}


\author*[1]{\fnm{Frederik} \spfx{vom} \sur{Ende} 
}\email{frederik.vom.ende@fu-berlin.de}

\author[2]{\fnm{Fereshte} \sur{Shahbeigi} 
}\email{fereshteshahbeigi@gmail.com}


\affil*[1]{\orgdiv{Dahlem Center for Complex Quantum
Systems}, \orgname{Freie Universität Berlin}, \orgaddress{\street{Arnimallee 14}, \city{Berlin}, \postcode{14195},
\country{Germany}}}

\affil[2]{\orgdiv{Research Center for Quantum Information
Institute of Physics}, \orgname{Slovak Academy of Sciences}, \orgaddress{\street{Dúbravská cesta 9}, \city{Bratislava}, \postcode{84511},
\country{Slovakia}}}


\abstract{This paper introduces the concept of a generating set for stochastic matrices---a subset of matrices whose repeated composition generates the entire set. Understanding such generating sets requires specifying the ``indivisible elements'' and ``building blocks'' within the set, which serve as fundamental components of the generation process. Expanding upon prior studies, we develop a framework that formalizes divisibility in the context of stochastic matrices. We provide a sufficient condition for divisibility that is shown to be necessary in dimension $n=3$, while for $n=2$, all stochastic matrices are shown to be divisible. Using these results, we construct generating sets for dimensions $2$ and $3$ by specifying the indivisible elements, and, importantly, we give an upper bound for the number of factors required from the generating set to produce the entire semigroup.}




\maketitle
\newpage
\section{Introduction}
\label{sec:intro}
Stochastic matrices, characterized by entry-wise non-negative matrices that sum to one in each column, play a leading role in studying stochastic processes, with applications ranging from Markov chains and population dynamics \cite{Norris97} to quantum channels \cite{Heinosaari12} and machine learning \cite{Privault24}. An intriguing aspect of stochastic matrices is understanding the minimal sets required to generate all possible transitions. This raises the question: can one identify a subset of stochastic matrices whose products reproduce the entire set of stochastic matrices? 

On the one hand, given that stochastic matrices form a semigroup, the above question is related to finding a set of generators for a semigroup. There are known results in this direction for some discrete and continuous groups, as well as for certain semigroups. For instance, the group ${\sf SL}(n,\mathbb Z)$, \mbox{$n\neq 4$} can be generated using two matrices \cite{GT93} with a generalization to special linear matrices over finite fields \cite{Ruskuc95}. On the other hand, the question of generating sets within the quantum realm has been posed and results for qubit (i.e., two-dimensional) channels specifically unital qubit channels, termed the universal set of channels \cite{BGN14}, have been established. Nonetheless, to our best knowledge,  there is a lack of similar studies for the set of stochastic matrices, either as a semigroup or as the classical counterparts of quantum channels.

The notion of generating sets for stochastic matrices vividly evokes the concept of divisibility, i.e., if a given stochastic matrix can be written as \textit{nontrivial} product of two of them. The importance of this relationship arises from the fact that any set of generators must contain some indivisible elements. Thus, it is crucial to identify the indivisible elements in the set of stochastic matrices. Surprisingly, even though (in)divisibility has been researched in the set of quantum channels \cite{Wolf08a}, entry-wise non-negative matrices \cite{RS74,dCG81}, and bistochastic matrices \cite{PHS98}, there is a lack of rigorous study of (in)divisibility in the set of stochastic matrices. In particular, we note that a divisible stochastic matrix is also divisible in the set of non-negative matrices \cite{RS74}, while the converse is not necessarily true.

Along similar lines, another essential related concept to address is that of \textit{building blocks}. Roughly speaking, a building block refers to a member of the set that is divisible, yet it is present in all of its (non-trivial) decompositions. 
The building block idea has since been applied to quantum channels to obtain simple sufficient criteria for channel divisibility \cite{vomEnde24ChannelDiv}. 

Our work tackles the question of generating sets of stochastic matrices. To this end, in parallel with the ideas outlined earlier, we study indivisibility and building blocks of stochastic matrices. In particular, we discuss that indivisibility is rooted in the fact that stochastic matrices form a semigroup rather than a group, and we provide a sufficient criterion for divisibility of stochastic matrices that is shown to also yield a necessary condition for stochastic matrices of dimensions $2$ and $3$. Identifying the indivisible elements and building blocks finally lets us present different generating sets for stochastic matrices of dimensions $2$ and $3$.

This paper is structured as follows.
In Section~\ref{sec_prelim} we recap and introduce this work's underlying concepts, such as divisibility and building blocks.
Our main results are then contained in Section~\ref{sec_main}.
More precisely, in Section~\ref{subsec_main_1} we adjust---in a non-trivial way---known divisibility results from general non-negative matrices to stochastic matrices.
Based on this, in Section~\ref{subsec_main_2} we present explicit generating sets of the stochastic matrices in two as well as three dimensions, where we also give an upper bound on how many prime factors one needs in the worst case.
Finally, we give a brief outlook in Section~\ref{sec_outlook}, where we also discuss how to generalize our construction to arbitrary finite dimensions.

\section{Preliminaries: Divisibility, Building Blocks, and Generating Sets}\label{sec_prelim}

Our starting point is a semigroup $(S,\circ)$ with identity $e$, also known as \textit{monoid}.
Studying generating sets of semigroups, i.e., subsets $G\subseteq S$ such that the semigroup $\langle G\rangle_s:=\{g_1\circ\ldots\circ g_m:m\in\mathbb N, g_1,\ldots,g_m\in G\}$ generated by $G$ is all of $S$, comes with the important notion of ``indivisible'' or
``prime'' elements\footnote{
In a commutative ring, the concept analogous to ``indivisibility'' is ``irreducibility'', while prime elements satisfy $p|ab$ $\Rightarrow$ $p|a$ or $p|b$ \cite{Sharpe87,AVL96}.
To avoid confusion, we will consistently use the term ``indivisible'' in this work.
}. These are the elements
which cannot be written as a non-trivial product of elements in $S$.
The term ``non-trivial'' here is key because one can of course always write $x=x\circ e$ or $x=(x\circ y)\circ y^{-1}$, assuming $y\in S$ has an inverse $y^{-1}$ which is also in $S$.
Thus, we first need to define the \textit{group of units}
\cite{HHL89,PHS98}
\begin{equation}\label{eq:S_inv}
S^{-1}:=\{x\in S:\exists_{y_x\in S}\ \ x\circ y_x=y_x\circ x=e\}\,.
\end{equation}
If such $y_x$ exists, it is necessarily unique which is why for all $x\in S^{-1}$ we may write $x^{-1}:=y_x$.

Note that the composition of any two elements of $S^{-1}$ is invertible and therefore belongs to $S^{-1}$. This guarantees the closure property and thus $(S^{-1},\circ)$ forms a group. This set is sometimes denoted by $H(S)$ to distinguish it from the elements of $S$ which are invertible in a more general sense (e.g., invertible as a matrix where the inverse is not in $S$).
In particular, if $S$ is itself a group, then $S^{-1}=S$, but it can of course happen that a semigroup has a trivial group of units, $S^{-1}=\{e\}$, e.g., $S=(\mathbb N,\cdot)$.
With all this in mind, the following definition is reasonable \cite[Def.~2.2]{PHS98}.
Be aware that, for the sake of simplicity, we will henceforth drop the symbol $\circ$ whenever convenient.\medskip

\begin{definition}\label{def_divisibility}
Given a monoid $(S,\circ)$, we say that $x\in S$ is \textit{indivisible} if for every decomposition $x=yz$ with $y,z\in S$ it holds that either $y$ or $z$---but not both---are in $S^{-1}$. Otherwise, $x$ is called \textit{divisible} (in $S$).
\end{definition}\medskip

Two remarks are in order:
First, this definition
covers prime numbers (when considering $S=(\mathbb N,\cdot)$), existing notions of ``prime'' matrices \cite{RS74,dCG81,PHS98,vandenHof99},
as well as the notion of indivisible ``quantum channels''\,\footnote{
For those familiar with quantum information theory: the connection of divisibility of channels and our Definition~\ref{def_divisibility} comes from the well-known fact that the group of units of the quantum channels are precisely the unitary channels
\cite[Proposition~4.31]{Heinosaari12}.
If one replaces complete positivity by positivity, then
the group of units is
generated by the unitary channels
together with the transposition map
\cite[Ch.~2, Corollary~3.2]{Davies76}.
}
\cite{Wolf08a}.
Second, with this definition, the indivisible elements of $S$ necessarily form a subset of $S\setminus S^{-1}$. Indeed, all $x\in S^{-1}$ can be written as $x=x\circ e$ so $x$ is divisible as it is a product of two elements in $S^{-1}$.
This is motivated by the fact that if $S^{-1}$ featured any indivisible elements, then unique (prime) factorizations could not exist.
As a direct consequence, one gets the following:\medskip
\begin{fact}
If $S$ is a group, then it does not admit any indivisible elements. In this sense, divisibility is a concept inherent to semigroups. 
\end{fact}\medskip
As we are interested in generating sets of semigroups, divisibility is not the only concept we will need. Indeed, it can happen that some $x\in S$ is divisible, but every decomposition $x=yz$ features $x$ itself (resp.~some element from $S^{-1}xS^{-1}$).
Before illustrating this with an example, recall that the $n\times n$ (column-)stochastic matrices, denoted by $s(n)$, are those $A\in\mathbb R^{n\times n}$ which satisfy $a_{jk}\geq 0$ for all $j,k=1,\ldots,n$, as well as $\sum_{j=1}^n a_{jk}=1$ for all $k=1,\ldots,n$.
Then
\begin{equation}\label{eq:ex_building_block_but_not_prime}
A=\begin{pmatrix}
1&1\\0&0
\end{pmatrix}\in s(2)
\end{equation}
is an example of a self-referential decomposition as $A$ is divisible ($A^2=A$) but $A$ cannot be decomposed into stochastic matrices neither of which is $A$ (up to permutation).
Hence, such elements of $S$ have to be contained in every generating set.
This leads to the following definition, which is fundamental to understanding generating sets.
\medskip\begin{definition}\label{def_building_block}
Given a monoid $(S,\circ)$, we say that $x\in S$ is a \textit{building block} (of $S$) if for every decomposition $x=yz$ with $y,z\in S$ it holds that $y\in x S^{-1}$, or $z\in S^{-1} x$, or both.
We write $B(S)$ (or just $B$, if $S$ is clear from the context) for the collection of all building blocks of $S$.
\end{definition}\medskip
Note that, by definition, the set of building blocks depends crucially on $S$ itself. For example, while the matrix in Eq.~\eqref{eq:ex_building_block_but_not_prime} is a building block of
$s(2)$, if we drop the normalization condition and define $S$ as the semigroup $\mathbb R_+^{2\times 2}$ of entry-wise non-negative $2\times 2$ matrices, then $A$ is not a building block anymore.
To clarify the relation between the previously given definitions, we state the following lemma.
The proof is straightforward and thus left to the reader.
\medskip\begin{lemma}\label{lemma_indiv_build_gen}
Given a monoid $(S,\circ)$, the following statements hold.
\begin{itemize}
\item[(i)] If $x\in S$ is indivisible, then $x\in B(S)$. 
\item[(ii)] If $G$ is a generating set of $S$, then $(S^{-1} x S^{-1})\cap G\neq\emptyset$ for all $x\in B(S)$.
\end{itemize}
\end{lemma}
\noindent 
The converse to Lemma~\ref{lemma_indiv_build_gen} (i) need not hold
as the matrix $A$ from~\eqref{eq:ex_building_block_but_not_prime} illustrates.
Moreover, if $S$ is a semigroup where every one-sided inverse is a two-sided inverse (e.g., any semigroup of real or complex square matrices), then $S^{-1}\subseteq B(S)$:
If $x=yz$ for some $y,z\in S$, then $e=(x^{-1}y)z$
(i.e., $z\in S^{-1}$)
which shows $y=xz^{-1}\in x S^{-1}$, as desired.
However, in this case $S^{-1} x S^{-1}=S^{-1}$ so Lemma~\ref{lemma_indiv_build_gen} states that every generating set of $S$ needs to contain at least one element from the group of units of $S$.

While the building blocks of a semigroup \textit{can} be a generating set (e.g., $S=(\mathbb N,\cdot)$ where $B(S)$ are the primes together with $1$) we point out that this is not true in general.
To illustrate this---as well as further relations between generating sets, building blocks, and the group of units---we present some simple examples.

\smallskip\begin{examplecustom}\label{ex_diff_prime_build}
\begin{itemize}
\item[(i)] Consider $S=([0,1],\cdot)$. It has no indivisible elements and its building blocks read $\{0,1\}$.
In particular, the building blocks do not generate $S$ as $\langle B(S)\rangle_s=B(S)\subsetneq S$.
However,
$\{0\}\cup[1-\varepsilon,1]$ 
is obviously a generating set for all $\varepsilon\in(0,1)$.
\item[(ii)] Next, we present a semigroup for which the building blocks are more than just the indivisible elements combined with the group of units, but
one needs all building blocks to generate $S$:
Consider
$$
S=\Big\{ \begin{pmatrix}
 1&0\\0&1
\end{pmatrix},\begin{pmatrix}
 1&1\\1&0
\end{pmatrix},\begin{pmatrix}
 1&1\\0&1
\end{pmatrix},\begin{pmatrix}
 1&1\\1&1
\end{pmatrix} \Big\},
$$
which is readily verified to be a subsemigroup of the boolean $2\times 2$ matrices
(i.e., after multiplication
all positive entries in the final result are replaced by ones) 
\cite{dCG81}.
Explicitly mapping out $S\cdot S$ shows that $S^{-1}=\{{\bf1}_2\}$, all elements of $S$ are divisible, and $\langle B\rangle_s=S$ where the building blocks $B$ of $S$ are given by
$$
B=\Big\{ \begin{pmatrix}
 1&0\\0&1
\end{pmatrix},\begin{pmatrix}
 1&1\\1&0
\end{pmatrix},\begin{pmatrix}
 1&1\\0&1
\end{pmatrix}\Big\}\subsetneq S\,.
$$
\item[(iii)] Finally, not all building blocks have to appear in every generating set: the set $S=\{{\bf1}_2,\sigma_x\}$ containing the identity and the Pauli matrix $\sigma_x$
satisfies
$S=S^{-1}=B(S)$ (because $S$ is itself a group). However, $G=\{\sigma_x\}\subsetneq B$ is of course a generating set for $S$.
\end{itemize} 
\end{examplecustom}\medskip

Importantly, Example~\ref{ex_diff_prime_build}~(i) constitutes an example where one can divide an element infinitely many times. This opposes, e.g., the discrete ordered semigroup $(\mathbb N,\cdot)$ where one arrives at an indivisible element in finitely many steps.
Indeed, the number of factors one needs to reconstruct a given element is the final concept we introduce in this section, which has previously been studied for the Lie group $\mathsf{SL}(2)$ \cite{HMM03,Mittenhuber02}:
\medskip\begin{definition}
Given a monoid $(S,\circ)$, a generating set $G\subseteq S$, and any $x\in S$, we define
$N_G(x)$ as the smallest possible number of elements from $G$ one needs to generate $x$.
Formally,
$$
N_G(x):=\min\{m:\exists_{x_1,\ldots,x_m\in G}\ \ x=x_1\circ\ldots\circ x_m\}\,.
$$
Moreover, we define $N_G(S):=\sup_{x\in S}N_G(x)$.
\end{definition}\medskip
Note that the above definition of course allows for repetition, i.e., it is possible to have $x_i=x_j$ for different $i$ and $j$. Moreover, while one always has $N_G(x)<\infty$ for all $x\in S$ by the definition of a generating set, it can happen that there does not exist a uniform bound on the number of needed factors; in this case $N_G(S)=\infty$.
For example, trivially, $N_S(S)=1$ for all semigroups, while $N_P(\mathbb N)=\infty$
(with $P$ the primes)
as there exist numbers with arbitrarily many prime factors.
The motivation for the function $N_G$ comes from an application point of view: 
\medskip\begin{remark}\label{rem_div_error_bound}
Assume $S$ is a set of ``physically plausible'' operations, and $G$ is the subset of operations one can implement in some experiment. Then, one wants to know (i) whether $G$  can generate everything, i.e., whether $\langle G\rangle_s=S$, and (ii) if so, whether one can bound how often operations from $G$ have to be prepared, i.e., does $N_G(S)<\infty$ hold and what is its value?
This second point is also what allows for simple error approximations: Assume one (e.g., an experimenter) does not have access to $G$ exactly, but only up to precision $\varepsilon>0$ (w.r.t.~some submultiplicative norm). Then every element from $S$ is guaranteed to be approximable with an error proportional
to\footnote{
Given any $g_1,\ldots,g_\ell\in G$, $\tilde g_1,\ldots,\tilde g_\ell\in S$, $\ell\in\mathbb N$ this follows from the standard telescope identity $\prod_{j=1}^\ell g_j-\prod_{j=1}^\ell\tilde g_j=\sum_{i=1}^\ell(\prod_{j=1}^{i-1}g_j)(g_j-\tilde g_j)(\prod_{j=i+1}^{\ell}\tilde g_j)$. Indeed, if the semigroup is bounded (with constant $M>0$), then $\|g_j-\tilde g_j\|<\varepsilon$ for all $j$ implies the estimate
$
\|\prod_{j=1}^\ell g_j-\prod_{j=1}^\ell\tilde g_j\|\leq M^{\ell-1}\ell\varepsilon\leq M^{\ell-1}N_G(S)\varepsilon
$,
by definition of $N_G(S)$.
For example, for $S=s(n)$ and $\|A\|:=\sup_{\|x\|_1=1}\|Ax\|_1$ one has $M=1$ so the error bound is $N_G(S)\varepsilon$.
}
$N_G(S)\varepsilon$.
\end{remark}\medskip
To illustrate the quantity $N_G(S)$ and show that it makes sense even if $S$ is not a semigroup but a group, consider the following example \cite[Ch.~4.2]{NC10}:
\medskip\begin{example}\label{ex_gen_set_unitary}
The set $G':=\{e^{it\sigma_z}:t\in\mathbb R\}\cup\{e^{it\sigma_y}:t\in\mathbb R\}$ (with $\sigma_y,\sigma_z$ the usual Pauli matrices)
satisfies $\langle G'\rangle_s=\mathsf{SU}(2)$ as well as $N_{G'}(\mathsf{SU}(2))=3$. That is, every $U\in\mathsf{SU}(2)$ can be written as a product of $3$ elements from $G'$, and the number $3$ is optimal.
This is a simple consequence of the Z-Y decomposition for qubits \cite[Theorem~4.1]{NC10}.
Interestingly, the generating set $G'$ can be made ``significantly'' smaller (in the sense that one of the $1$-parameter groups can be replaced by a single unitary) with the trade-off of requiring at most $5$ factors instead of $3$.
For this consider
 $G:= \{e^{it\sigma_z}:t\in\mathbb R\}\cup \{e^{i\pi\sigma_y/4}\}$; we claim that $\langle G\rangle_s=\mathsf{SU}(2)$ as well as $N_G(\mathsf{SU}(2))=5$.
Indeed, based on the Z-Y decomposition one for all $t,\phi_x,\phi_y\in\mathbb R$ finds
\begin{align*}
 &\begin{pmatrix}
 \cos(t)e^{i\phi_x}&\sin(t)e^{i\phi_y}\\
 -\sin(t)e^{-i\phi_y}&\cos(t)e^{-i\phi_x}
\end{pmatrix}\\
&\qquad\qquad= e^{i({\phi_x+\phi_y})\sigma_z/2}e^{it\sigma_y}e^{i({\phi_x-\phi_y})\sigma_z/2} \\
&\qquad\qquad=e^{i({2\phi_x+2\phi_y+5\pi})\sigma_z/4}e^{i\pi\sigma_y/4}e^{i(t+\pi/2)\sigma_z}e^{i\pi\sigma_y/4}e^{i({2\phi_x-2\phi_y+\pi})\sigma_z/4} \,.
\end{align*}
As the matrix on the left-hand side is a generic element of $\mathsf{SU}(2)$ this shows $\langle G\rangle_s=\mathsf{SU}(2)$ as well as $N_G(\mathsf{SU}(2))\leq 5$.
The final step is to see that, e.g.,
$$
\begin{pmatrix}
 \cos(\pi/8)&\sin(\pi/8)\\-\sin(\pi/8)&\cos(\pi/8)
\end{pmatrix}\in\mathsf{SU}(2)
$$
cannot be written as a product of $4$ or less elements from $G$ as a straightforward computation shows.
\end{example}\medskip

\noindent The question of reachability involving discrete control operators from a Lie-theoretic perspective has also been looked at in the context of quantum walks \cite{AD09,AD12}. Transferring their idea to the above example lets one arrive at the Lie algebra generated by $\{\sigma_z, e^{i\pi\sigma_y/4}\sigma_ze^{-i\pi\sigma_y/4}$, $ (e^{i\pi\sigma_y/4})^2\sigma_z(e^{-i\pi\sigma_y/4})^2,\ldots\}$, that is, the continuous direction one has access to is iteratively conjugated by the discrete control operator. However, this does not say anything about the number of factors one needs.

Either way, the trade-off between the ``size'' of the generating set and the number of needed factors in Example~\ref{ex_gen_set_unitary} can also occur in a more extreme way.
To see this, we go back to the Example~\ref{ex_diff_prime_build}~(i), i.e., $S=([0,1],\cdot)$ and $G_\varepsilon:=\{0\}\cup[1-\varepsilon,1]$, $\varepsilon\in(0,1)$. Obviously, $N_G(S)=\infty$ as $N_{G_\varepsilon}\left((1-\varepsilon)^n\right)=n$ for all $n\in\mathbb N$. 
On the other hand, if one considers $G:=[0,\varepsilon_1]\cup[\varepsilon_2,1]$ for some $0<\varepsilon_1<\varepsilon_2<	1$, then $N_G(S)=\lceil \frac{\ln(\varepsilon_1)}{\ln(\varepsilon_2)} \rceil$ as this is the smallest possible number $m$ such that $\varepsilon_2^m\leq \varepsilon_1$.

\section{Main Results}\label{sec_main}
Having set the stage, we now investigate divisibility and generating sets for the semigroup of stochastic matrices $s(n)$ in the following two sections.

\subsection{Divisibility of stochastic matrices}\label{subsec_main_1}
First recall that $s(n)^{-1}\simeq S_n$ (where $S_n$ is the symmetric group of order $n$), i.e., the group of units corresponding to $s(n)$ are precisely the permutation matrices
\cite[Section~2.3.4]{BP79}.
With this, Definitions~\ref{def_divisibility} \& \ref{def_building_block} become the following:
$A\in s(n)$ is called indivisible (resp., a building block) if for every decomposition $A=A_1A_2$ into $A_1,A_2\in s(n)$ it holds that either $A_1$ or $A_2$ is a permutation matrix (resp., $A_1$ or $A_2$ or both are $A$ up to 
permutation matrices). Henceforth---unless specified otherwise---when writing ``(in)divisible'', we mean ``(in)divisible in $s(n)$''.

Obviously, given any $A\in s(n)$ and permutations\footnote{
Given any 
$\pi\in S_n$ we define the induced permutation matrix via $\underline{\pi}:=\sum_je_je_{\sigma(j)}^T$. The usual identities $\underline{\sigma\circ\tau}=\underline{\tau}\cdot\underline{\sigma}$, $(\underline{\sigma}x)_j=x_{\sigma(j)}$, and $(\underline{\sigma}A\underline{\tau})_{jk}=A_{\sigma(j)\tau^{-1}(k)}$ hold.
}
$\pi,\tau\in S_n$, $A$ is divisible if and only if $\underline{\pi}A\underline{\tau}$ is divisible.
Factoring out this group action of $S_n\times S_n$ on $s(n)$ yields equivalence classes of (in)divisible matrices.
Now, as first used in \cite{RS74}, an important tool in analyzing the divisibility of stochastic matrices is the usual sign function 
\begin{equation}\label{eq:sign-definition}
 \sgn(x)=\begin{cases}
 \frac{x}{|x|}\quad &x\neq0\\
 0\quad &x=0
 \end{cases}.
\end{equation}
This function is applied to a matrix entry-wise. In particular, $A\in s(n)$ is a permutation matrix if and only if $\sgn(A)
$ is a permutation matrix.
Also, a readily verified identity which holds for all $B,C\in\mathbb R_+^{n\times n}$ is 
\begin{equation}\label{eq:sgn_multiplicative}
 \sgn(BC)=\sgn(\sgn(B)\sgn(C))\,.
\end{equation}
In particular, this implies
$\sgn(\underline{\pi}A\underline{\tau})=\underline{\pi}\sgn(A)\underline{\tau}$
and, as already observed in \cite{dCG81}, this allows for a handy necessary criterion for indivisibility by considering the finite set  $\sgn(s(n)):=\{\sgn(A):A\in s(n)\}$ with $(2^n-1)^n$ elements instead:
\medskip\begin{lemma}\label{lemma_indiv}
Let $A\in s(n)$ and assume that $\sgn(A)$ is indivisible, that is, $\sgn(A)=\sgn(\sgn(B)\sgn(C))$ for any $B,C\in s(n)$ implies that $B$ or $C$ is a permutation matrix.
Then $A$ is indivisible.
\end{lemma}\smallskip
As an example, the matrix
\begin{equation}\label{eq:prime_3}
 \begin{pmatrix}
 0&1&1\\1&0&1\\1&1&0
\end{pmatrix}
\end{equation}
is indivisible in the sense of Lemma~\ref{lemma_indiv} (and it is even indivisible in $\mathbb R_+^{3\times3}$), cf.~\cite{dCG81,RS74}.
Therefore, by Lemma~\ref{lemma_indiv}, every $A\in s(3)$ for which $\sgn(A)$ is equal to~\eqref{eq:prime_3} is itself indivisible.
This readily transfers to the equivalence class of $A$ under $S_3\times S_3$, i.e., if $\sgn(A)$ is of one of the following forms, then $A$ is indivisible:
$$\!
 \begin{pmatrix}
 0&1&1\\1&0&1\\1&1&0
\end{pmatrix},
 \begin{pmatrix}
 0&1&1\\1&1&0\\1&0&1
\end{pmatrix},
 \begin{pmatrix}
 1&0&1\\0&1&1\\1&1&0
\end{pmatrix},
 \begin{pmatrix}
 1&0&1\\1&1&0\\0&1&1
\end{pmatrix},
 \begin{pmatrix}
 1&1&0\\0&1&1\\1&0&1
\end{pmatrix},
 \begin{pmatrix}
 1&1&0\\1&0&1\\0&1&1
\end{pmatrix}.
$$
Moreover, the fact that this matrix 
is prime
is also in line with the fact that $\frac12\cdot$%
Eq.~\eqref{eq:prime_3} is an example of a doubly-stochastic 
matrix which cannot be written as the product of T-transforms \cite{MKS84}.\medskip

Either way, the sign function turns out to also be useful for proving sufficient criteria for divisibility;
the following lemma is an adjustment of \cite[Theorem~2.4]{RS74} from $\mathbb R_+^{n\times n}$ to $s(n)$:

\medskip\begin{lemma}\label{lemma_cond_div}
Let $A\in s(n)$, $n\geq 2$ such that $\sgn(A)e_i\geq\sgn(A)e_k$
for some $i,k\in\{1,\ldots,n\}$, $i\neq k$. Here $\{e_j\}_j$ denotes the standard basis. Then, $A$ is divisible and there exists a corresponding decomposition $A=BC$ into $B,C\in s(n)$ neither of which are permutation matrices, such that $B$ differs from $A$ in at most the $i$-th column, and
\begin{equation}\label{eq:C_twodim_block}
 C=\underline{\pi}\Big(\begin{pmatrix}
 1&a\\0&1-a
\end{pmatrix}\oplus{\bf1}_{n-2}\Big)\underline{\tau}
\end{equation}
for some $a\in(0,1]$, $\pi,\tau\in S_n$.
Moreover, if $Ae_i\neq e_j$ for any $j=1,\ldots,n$, then
the number of zero entries in $B$ is strictly larger than the number of zeros in $A$.
\end{lemma}
\begin{proof}
 The first part relies on an auxiliary function, defined as
 \begin{align*}
 \varepsilon:\Delta^{n-1}\times\Delta^{n-1}&\to\mathbb R_+\\
 \varepsilon(v,w)&:=\min_{\{ j=1,\ldots,n : w_j>0 \}}\frac{v_j}{w_j}
 \end{align*}
 where, $\Delta^{n-1}:=\{v\in\mathbb R_+^n:\sum_{j=1}^nv_j=1\}$ denotes the standard probability simplex in $n$ dimensions. This function satisfies the following properties:
 \begin{itemize}
 \item If $\sgn(v)\geq\sgn(w)$, then $\varepsilon(v,w)>0$ (as $w_j>0$ implies $v_j>0$).
 \item For all $v,w\in\Delta^{n-1}$ one has $\varepsilon(v,w)\leq 1$:
 \begin{align*}
 1=\sum_{j=1}^nv_j\geq\sum_{j=1}^n\sgn(w_j)v_j&=\sum_{\{ j=1,\ldots,n : w_j>0 \}}w_j\frac{v_j}{w_j}\\
 &\geq\varepsilon(v,w)\sum_{\{ j=1,\ldots,n : w_j>0 \}}w_j=\varepsilon(v,w)
 \end{align*}
 \item
 $\varepsilon(v,w)=1$ if and only if $v=w$:
 While the ``only if''-part is trivial, for the ``if''-part note that $\varepsilon(v,w)=1$ implies $v_j\geq w_j$ for all $j$. This, in turn, for all $j$ shows $v_j=1-\sum_{k\neq j}v_k\leq1-\sum_{k\neq j}w_k=w_j$. Therefore, altogether, we get $v_j=w_j$.

 \end{itemize}
With this in mind, we prove the lemma in two steps. First, we prove the case where $\varepsilon(Ae_i,Ae_k)=1$, and after that, we prove that the lemma holds when this equality is violated.
\begin{itemize}
 \item[1.] Let $\varepsilon(Ae_i,Ae_k)=1$. As seen before, this means $Ae_i=Ae_k$. We define
 \begin{align*}
 B&:=\begin{cases}
 A-Ae_ie_i^T+e_ke_i^T&\substack{\text{ if }A-Ae_ie_i^T+e_ie_i^T\text{ is}\\\text{a permutation matrix}}\\
 A-Ae_ie_i^T+e_ie_i^T&\text{ else}
 \end{cases}
 \end{align*}
 and $C:={\bf1}-e_i e_i^T+e_k e_i^T\in s(n)$, so $C$ is of the desired form~\eqref{eq:C_twodim_block} with $a=1$. 
 Using $Ae_i=Ae_k$, a straightforward computation shows that $A=BC$ in any case. 
 Moreover, $B$ equals $A$ aside from, possibly, the $i$-th column which now is $e_i$ or $e_k$. This choice ensures that $B\in s(n)$ is not a permutation matrix, meaning $A$ is indeed divisible.
 \item[2.] Now for the case where $\varepsilon:=\varepsilon(Ae_i,Ae_k)<1$. Then $Ae_i-\varepsilon Ae_k\geq 0$ by definition of $\varepsilon$, which in turn implies that $\frac{Ae_i-\varepsilon Ae_k}{1-\varepsilon}\in\Delta^{n-1}$. Thus,
 $A=BC$ where $B:=A-\frac{\varepsilon}{1-\varepsilon}A(e_k e_i^T-e_i e_i^T)$ is equal to $A$ except for the $i$-th column which now reads $\frac{Ae_i-\varepsilon Ae_k}{1-\varepsilon}$, and $C:={\bf1}+\varepsilon(e_k e_i^T-e_i e_i^T)$. Moreover, one shows that all zeros of $Ae_i$ are also a zero of $\frac{Ae_i-\varepsilon Ae_k}{1-\varepsilon}$.
\end{itemize}
It remains to prove that $B$ has more zeros if $Ae_i\neq e_j$ for all $j=1,\ldots,n$. To see this, note that in Case 1~the vector $Ae_i$ has at least two non-vanishing entries whereas $Be_i$---which is either $e_i$ or $e_k$---has only one.
In case 2~the vector $\frac{Ae_i-\varepsilon Ae_k}{1-\varepsilon}$ has at least one more zero than $Ae_i$ by definition of $\varepsilon$.
Altogether, this concludes the proof.
\end{proof}

\noindent Note that the matrix $C$ in Eq.~\eqref{eq:C_twodim_block} (up to re-scaling) has been called ``elementary matrix'' \cite{CK03}.
Either way, we get the following result as a direct consequence of Lemma~\ref{lemma_cond_div}.
\medskip\begin{corollary}
 A stochastic matrix $A\in s(n)$ is divisible if there exists an $i\in\{1,\ldots,n\}$ such that $Ae_i>0$, i.e., $\sgn(A)e_i=(1,\ldots,1)^T$.
\end{corollary}\medskip
This is a comparable, but stronger version of a result from quantum information theory which states that every full-rank quantum channel is divisible~\cite[Theorem~11]{Wolf08a}. Our corollary implies that, like with quantum channels, the subset of indivisible stochastic matrices is of measure zero in $s(n)$.

Next, we showcase the strength of this lemma by applying it to the smallest non-trivial dimensions $n=2,3$. For $n=2$, using Lemma~\ref{lemma_cond_div} one directly arrives at the following:
\medskip\begin{fact}
 All two-dimensional stochastic matrices are divisible.
\end{fact}\medskip
In other words, all elements of $s(2)$ are non-trivial products of elements of $s(2)$ (i.e., one does not need $\mathbb R_+^{2\times 2}\setminus s(2)$ for divisibility), which thus generalizes the known results of divisibility from $\mathbb R_+^{2\times 2}$ \cite{RS74,dCG81} to $s(2)$.
For $n=3$, while we gave an example of a family of indivisible matrices in $s(3)$ (Eq.~\eqref{eq:prime_3} ff.), at this point, we can prove that all primes are of this form by applying Lemma~\ref{lemma_cond_div}.

\medskip\begin{theorem}\label{thm_primes_s3}
Given $A\in s(3)$, $A$ is indivisible if and only if there exist
$\pi,\tau\in S_3$ such that
$$
\underline{\pi} \sgn(A)\underline{\tau}=\begin{pmatrix}
 0&1&1 \\ 1&0&1 \\ 1&1&0
\end{pmatrix}\,.
$$
\end{theorem}
\begin{proof}
 ``$\Leftarrow$'': Follows from Lemma~\ref{lemma_indiv} ff.
 ``$\Rightarrow$'': One readily verifies that there are only two elements of (equivalence classes of) $\sgn(s(3))$ where none of the columns are comparable in the sense of the assumption of Lemma~\ref{lemma_cond_div}:
 $$
 \begin{pmatrix}
 1&0&0\\0&1&0\\0&0&1
 \end{pmatrix}\quad\text{ and }\quad\begin{pmatrix}
 0&1&1\\1&0&1\\1&1&0
 \end{pmatrix}\,.
 $$
 The first matrix corresponds to $s(n)^{-1}\simeq S_n$ which is always divisible, meaning if (the $\sgn$ of) $A$ is not equivalent to the second matrix, it is divisible. This concludes the proof.
\end{proof}

\noindent Interestingly, this structure does not generalize to higher dimensions.
While there exist indivisible stochastic matrices of this form, e.g.
\cite{BHM76}
\begin{equation}\label{eq:indiv_s4}
\frac17\begin{pmatrix}
0&1&1&5\\
5&0&1&1\\
1&5&0&1\\
1&1&5&0
\end{pmatrix},
\end{equation}
the following decomposition can (in a slightly different form) be found in \cite{RS74}:
\begin{align*}
\begin{pmatrix}
0&\frac12&\frac12&\frac12\\
\frac13&0&\frac14&\frac14\\
\frac13&\frac14&0&\frac14\\
\frac13&\frac14&\frac14&0
\end{pmatrix}=\begin{pmatrix}
1&0&0&0\\
0&0&\frac12&\frac12\\
0&\frac12&0&\frac12\\
0&\frac12&\frac12&0\\
\end{pmatrix}
\begin{pmatrix}
0&\frac12&\frac12&\frac12\\
\frac13&\frac12&0&0\\
\frac13&0&\frac12&0\\
\frac13&0&0&\frac12\\
\end{pmatrix}\in s(4)
\end{align*}
In other words, in four and more dimensions considering $\sgn(A)$ may be inconclusive when examining divisibility.

A characterization of indivisible elements of $\mathbb R_+^{n\times n}$ and of the doubly stochastic matrices is given in \cite[Theorem~4.1 \& 5.1]{PHS98}.
Clearly, the indivisible stochastic elements of $\mathbb R_+^{n\times n}$ are also prime in $s(n)$, 
and the divisible doubly-stochastic matrices are also divisible in $s(n)$.
\medskip

\subsection{Generating sets of $s(2)$ and $s(3)$}\label{subsec_main_2}

After these results on divisibility, let us move on to the task of finding non-trivial generating sets of the stochastic matrices.
Recall that this amounts to finding a---in some sense ``small''---set $G\subseteq s(n)$
such that the semigroup generated by $G$ is equal to
$s(n)$.
For the remainder of this section we look at the generating-set problem in two and three dimensions---which is already highly non-trivial---and at the end, we discuss how to generalize our construction to arbitrary finite dimensions. For $n=2$, one finds the following: 

\medskip\begin{proposition} \label{prop:G_of_s2}
The semigroup generated by
\begin{equation}\label{eq:prop_G_of_s2}
G:=\Big\{\begin{pmatrix}
0&1\\
1&0
\end{pmatrix}\Big\}\cup\;{\rm conv}\Big\{\mathbf1,\begin{pmatrix}
1&1\\
0&0
\end{pmatrix}\Big\}
\end{equation}
is equal to $s(2)$ with $N_G(s(2))=4$, i.e., $G^4=s(2)$ while $G^3\subsetneq s(2)$.
\end{proposition}
 \begin{proof}
 Let $A\in s(2)$ be parameterized as
 $$
 A= \begin{pmatrix}
 1-a&b\\a&1-b
 \end{pmatrix}\, ,
 $$
for some $a,b\in[0,1]$. If $a=0$, then $A\in G$. For $a=1$, $N_G(A)=2$ because
 $$
 \begin{pmatrix}
 0&b\\1&1-b
 \end{pmatrix}=\begin{pmatrix}
 0&1\\1&0
 \end{pmatrix}\begin{pmatrix}
 1&1-b\\0&b
 \end{pmatrix}\in G\cdot G\,.
 $$
 Now assume $a\in(0,1)$. We distinguish two cases.
 \begin{itemize}
 \item[1.] 
 If $a+b\leq 1$, then $\frac{b}{1-a}\in[0,1]$ and the following decomposition is well defined:
 $$
 A=\sigma_x\begin{pmatrix}
 1&a\\0&1-a
 \end{pmatrix}\sigma_x\begin{pmatrix}
 1&\frac{b}{1-a}\\0&\frac{1-a-b}{1-a}
 \end{pmatrix}\, ,
 $$
 \item[2.] If $a+b> 1$, we apply $\sigma_x$ 
from the left to get a matrix where the sum of the off-diagonals does
 not exceed $1$. Then, by the previous step
\begin{align*}A=\sigma_x(\sigma_xA)&=\sigma_x\Big(\sigma_x\begin{pmatrix}
 1&1-a\\0&a
 \end{pmatrix}\sigma_x\begin{pmatrix}
 1&\frac{1-b}{a}\\0&\frac{a+b-1}{a}
 \end{pmatrix}\Big)\\
 &=\begin{pmatrix}
 1&1-a\\0&a
 \end{pmatrix}\sigma_x\begin{pmatrix}
 1&\frac{1-b}{a}\\0&\frac{a+b-1}{a}
 \end{pmatrix}\,.
\end{align*} 
 \end{itemize}
This also shows $N_G(s(2))\leq 4$. For the converse inequality, we have to find some $A\in s(2)$ such that $N_G(A)\geq 4$, i.e., $A$ cannot be written as a product of at most three elements from $G$. Indeed,
$A:=\frac23\cdot{\bf1}+\frac13\cdot\sigma_x\in s(2)$
is such a matrix: The only product in $G^3$ for which all entries can be non-zero is
$$
\begin{pmatrix}
    1&1-\lambda\\0&\lambda
\end{pmatrix}\sigma_x\begin{pmatrix}
    1&1-\mu\\0&\mu
\end{pmatrix}\, ,
$$
with $\lambda,\mu\in[0,1]$, hence this is only way we could express $A$. But if $A$ were of this form, then $\frac23=A_{22}=\lambda(1-\mu)\leq\lambda=A_{21}=\frac13$, a contradiction. 
 \end{proof}
\begin{remarkcustom}
\begin{itemize}
    \item[(i)] Let us quickly interpret this result from the perspective of quantum physics. The generating set $G$ in Proposition~\ref{prop:G_of_s2} can be realized by combining a bit-flip ($\sigma_x$) with the classical amplitude-damping noise, that is, with the Markovian dynamics
\begin{equation}\label{eq:Markovian_s2}
    t\mapsto \exp\Big(t \begin{pmatrix}
0&1\\0&-1
\end{pmatrix} \Big)=\begin{pmatrix}
1&1-e^{-t}\\0&e^{-t}
\end{pmatrix}.
\end{equation}
More precisely, the previous proof shows that a generic stochastic $2\times 2$-matrix can be implemented by interleaving the dynamics~\eqref{eq:Markovian_s2} with a precisely timed bit-flip, possibly with a final bit-flip at the end.
\item[(ii)] Of course, there are uncountably many generating sets of $s(n)$. The advantage of the set $G$ from Proposition~\ref{prop:G_of_s2} is that, for $n=2$, it offers a nice trade-off between the size of the generating set on the one hand, and the number of required factors on the other hand, while also allowing for a clear physical interpretation.
Indeed, recalling the example from the end of Section~\ref{sec_prelim}, excluding sub-intervals from the convex hull in~\eqref{eq:prop_G_of_s2} will make $N_G(s(2))$ grow drastically, possibly to infinity.
Similar effects have been observed for qubit channels before \cite{BGN14}. 
One ``less applied'' construction for a smaller generating set is to limit the convex-combination coefficients in \eqref{eq:prop_G_of_s2} to the following (measure-zero) generating set of $([0,1],\cdot)$:
Starting from the Cantor set $\mathcal C$---which is uncountable, of measure zero, and satisfies $\mathcal C+\mathcal C=[0,2]$ \cite[pp.~164-165]{Schroeder91}---the set $\bigcup_{n\in\mathbb N_0}(n+\mathcal C)$ finitely generates $(\mathbb R_+,+)$. Applying $x\mapsto e^{-x}$ to this turns it into a generating set of $([0,1],\cdot)$ with the desired properties.
\end{itemize}
\end{remarkcustom}\medskip

Moving on to higher dimensions, we already saw in Section~\ref{subsec_main_1} that $s(n)$ for $n\geq 2$ contains indivisible elements.
Of course, these have to be featured in any generating set $G$.
For $n=3$ this is essentially ``all one needs'' (in the sense that 
the relative interior of the following set $G$ is precisely the indivisible elements of $s(3)$ modulo the group action $S_3\times S_3$).
The idea of our proof will be to make repeated use of Lemma~\ref{lemma_cond_div}, either until we arrive at an elementary matrix or until we hit a prime.
 
\medskip\begin{theorem}\label{thm_s3}
The semigroup generated by
\begin{align}\label{eq:s3_generators}
G:=\Big\{\begin{pmatrix}
0&1&0\\
1&0&0\\
0&0&1
\end{pmatrix}\Big\}\cup\, {\rm conv}\Big\{\mathbf1,\begin{pmatrix}
1&0&1\\
0&1&0\\
0&0&0
\end{pmatrix}
\,,\,\begin{pmatrix}
1&0&1\\
0&0&0\\
0&1&0
\end{pmatrix}
\,,\,\begin{pmatrix}
0&0&1\\
1&0&0\\
0&1&0
\end{pmatrix}\Big\}
\end{align}
equals $s(3)$ with $G^{20}=s(3)$, i.e., $N_G(s(3))\leq 20$ as the number of factors from $G$ needed to obtain any stochastic matrix is upper bounded by 20.
\end{theorem}
\begin{proof}
First, let us show that $\langle G\rangle_s=s(3)$. For this, note that every permutation matrix can be constructed from at most two elements from $G$: in cycle notation, $\underline{(1\,2)(3)},\underline{(1)(2)(3)},\underline{(1\,3\,2)}\in G$ and the remaining permutation matrices are in $G\cdot G$ because $\underline{(1)(2\,3)}=\underline{(1\,2)(3)}\cdot\underline{(1\,3\,2)}$, $\underline{(1\,2\,3)}=\underline{(1\,3\,2)}^2$, and $\underline{(1\,3)(2)}=\underline{(1\,3\,2)}\cdot\underline{(1\,2)(3)}$.
Thus, in slight abuse of notation, $N_G(S_3)=2$.

With this, given any $A\in s(3)$, our basic idea is to consider two cases: 1.~If Lemma~\ref{lemma_cond_div} does not yield a decomposition of $A$ into stochastic matrices that feature more zeros than $A$, then we have to show explicitly that $A\in \langle G\rangle_s$.
2.~If Lemma~\ref{lemma_cond_div} does yield such a decomposition $A=BC$, then it suffices to show that
$B,C\in\langle G\rangle_s$. But $C\in\langle G\rangle_s$ holds because the building blocks
\begin{equation}\label{eq:buildingblocks}
    \Big\{\begin{pmatrix}
 1&0&a\\0&1&0\\0&0&1-a
\end{pmatrix}:a\in[0,1]\Big\}
\end{equation}
are in $G$,
and because all permutation matrices are in $\langle G\rangle_s$. 

Altogether, what this shows is that we can disregard all elements of $s(3)$ that satisfy the following criteria from Lemma~\ref{lemma_cond_div} (which guarantee a decomposition $A=BC$ such that $B$ features more zeros than $A$): there exist $i\neq k$ such that $\sgn(A)e_i\geq\sgn(A)e_k$, and $Ae_i\neq e_j$ for all $j$.
Crucially, as $\sgn(A)$ contains all the relevant information for checking this, we can, equivalently, 
consider the quotient space $\sgn(s(3))/(S_3\times S_3)$ (as opposed to $\sgn(s(3))$, because every permutation matrix is in $\langle G\rangle_s$).
Eliminating those elements of $\sgn(s(3))/(S_3\times S_3)$ which satisfy the above criteria, we are left with the following five equivalence classes:
\begin{equation}\label{eq:remaining_matrices}
    {\bf1},
\begin{pmatrix}
 0 & 1 & 1 \\
 1 & 0 & 1 \\
 1 & 1 & 0 
 \end{pmatrix},
 \begin{pmatrix}
 0 & 0 & 0 \\
 0 & 0 & 0 \\
 1 & 1 & 1
 \end{pmatrix}, \begin{pmatrix}
 0 & 0 & 1 \\
 0 & 0 & 1 \\
 1 & 1 & 0
 \end{pmatrix},
 \begin{pmatrix}
 0 & 0 & 0 \\
 0 & 0 & 1 \\
 1 & 1 & 0
 \end{pmatrix}
\end{equation}
Some comments on these matrices are in order:
 \begin{itemize}
 \item The first element represents all permutation matrices, which are all in $\langle G\rangle_s$ as we already showed.
 \item The second element represents precisely the indivisible elements of $s(3)$ from Theorem~\ref{thm_primes_s3}.
These are in $\langle G\rangle_s$ 
because the convex hull in~\eqref{eq:s3_generators} is equal to
$$
\mathcal C:=\Big\{ \begin{pmatrix}
 a&0&1-c\\1-a&b&0\\0&1-b&c
\end{pmatrix} : a,b,c\in[0,1],a\geq b\geq c \Big\}\, ;
$$
This follows directly from the convex decomposition
\begin{align*}
  &   \begin{pmatrix}
 a&0&1-c\\1-a&b&0\\0&1-b&c
\end{pmatrix}=\\
&\ =c\,{\bf1}+(b-c)\begin{pmatrix}
1&0&1\\
0&1&0\\
0&0&0
\end{pmatrix}
+(a-b)\begin{pmatrix}
1&0&1\\
0&0&0\\
0&1&0
\end{pmatrix}
+(1-a)\begin{pmatrix}
0&0&1\\
1&0&0\\
0&1&0
\end{pmatrix}
\end{align*}
with $1\geq a\geq b\geq c\geq 0$.
In particular, for every $A\in s(3)$ with
$$
\sgn(A)\in\Big\{\underline{\pi}\begin{pmatrix}
 0&1&1\\1&0&1\\1&1&0
\end{pmatrix}\underline{\tau}:\pi,\tau\in S_3\Big\}
$$
there certainly exist permutations $\pi_A,\tau_A\in S_3$ such that $\underline{\pi_A}A\underline{\tau_A}\in\mathcal C$.
 \end{itemize}
Also the set $\mathcal C$ already covers the last two elements of~\eqref{eq:remaining_matrices} up to permutations
(choose $a=1$, $b\in[0,1]$, $c=0$).
This leaves the third element of~\eqref{eq:remaining_matrices}, which is in $\langle G\rangle_s$ because
\begin{equation}\label{eq:decomp_111}
\begin{pmatrix}
 0 & 0 & 0 \\
 0 & 0 & 0 \\
 1 & 1 & 1
 \end{pmatrix}= \begin{pmatrix}
 0 & 1 & 0 \\
 0 & 0 & 0 \\
 1 & 0 & 1
 \end{pmatrix}\begin{pmatrix}
 0 & 1 & 0 \\
 0 & 0 & 0 \\
 1 & 0 & 1
 \end{pmatrix}\in\langle G\rangle_s^2=\langle G\rangle_s\,.
 \end{equation}
Altogether, this proves the existence of a decomposition of any $A\in s(3)$ in terms of~\eqref{eq:s3_generators}.

Next is the upper bound 20 on the number of factors.
Recall that, starting from some $A\in s(3)$ our strategy was to apply Lemma~\ref{lemma_cond_div} to decompose $A=BC$ (assuming $B$ has more zeros than $A$) as often as possible, and whatever we are left with we proved explicitly to be in $\langle G\rangle_s$.
This results in a decomposition
$
A=B_1C_1=(B_2C_2)C_1=\ldots=B_mC_m\ldots C_1
$ for some $m\in\mathbb N_0$.
Here, $C_j$ is a building block from Eq.~\eqref{eq:buildingblocks} (up to permutation). In other words, there exist permutations $\sigma_1,\ldots,\sigma_m,\tau_1,\ldots,\tau_m\in S_3$ as well as $g_1,\ldots,g_m\in G$ such that
$
A=B_m\prod_{j=1}^m\underline{\tau_j}g_j\underline{\sigma_j}
$. Combining products $\underline{\sigma_j}\,\underline{\tau_{j-1}}$ into one permutation matrix $\underline{\pi_j}$ (with $\pi_1:=\sigma_1$) yields the decomposition
\begin{equation}\label{eq:maxfactors_A}
    A=B_m\underline{\tau_m}\prod_{j=1}^m g_j\underline{\pi_j}\,.
\end{equation}
As established before, $N_G(S_3)=2$ which guarantees that $g_j\underline{\pi_j}\in G^3$ for all $j$. Thus Eq.~\eqref{eq:maxfactors_A} yields the upper bound for $N_G(A)\leq N_G(B_m\underline{\tau_m})+3m$.
For bounding $N_G(B_m\underline{\tau_m})$, in turn, we will consider the following two scenarios:
\begin{itemize}
    \item[1.] Lemma~\ref{lemma_cond_div} was applied $6$ times, i.e., $m=6$. This is the maximum possible amount because after each step the resulting matrix gains a zero---but the number of non-zero elements (which is at most $3^2=9$) cannot be less than $3$ because of stochasticity. Hence, in this case, all entries of $A$ were non-zero and $B_6\underline{\tau_6}$ has exactly three non-zero entries. Therefore, up to permutations, $B_6\underline{\tau_6}$ is either ${\bf1}$, 
    \begin{equation}\label{eq:g_or_g2}
g=\begin{pmatrix}
1&0&1\\
0&0&0\\
0&1&0
\end{pmatrix}(\in G)\quad \text{or}\quad \begin{pmatrix}
1&1&1\\
0&0&0\\
0&0&0
\end{pmatrix}=\begin{pmatrix}
1&0&1\\
0&0&0\\
0&1&0
\end{pmatrix}^2.
    \end{equation}
    We claim that the latter two cases are not possible, i.e.~$B_6\underline{\tau_6}$ has to be a permutation which would show
    $$
    N_G(A)\leq N_G(B_m\underline{\tau_m})+3m\leq N_G(S_3)+3m\leq 2+3\cdot 6=20\,.
    $$
    Indeed, if $B_6$ were one of the matrices from Eq.~\eqref{eq:g_or_g2}, then $B_6\underline{\tau_6}$ had at least one row of zeros so the same would be true for $A=B_6\underline{\tau_6}\prod_{j=1}^6 g_j\underline{\pi_j}$.
    But this contradicts the fact that all entries of $A$ were non-zero.
    \item[2.] Lemma~\ref{lemma_cond_div} was applied less than 6 times, i.e., $m\leq 5$. In this case $\sgn(B_m)$ could, up to permutation, be any matrix from Eq.~\eqref{eq:remaining_matrices}.
    As in case 1.~if $B_m\underline{\pi_m}$ were matrix~\eqref{eq:decomp_111}, then $A=B_m\underline{\tau_m}\prod_{j=1}^m g_j\underline{\pi_j}$ would be of the same form; hence
    $A=\underline{\pi_1}g_1g_2\underline{\pi_2}$ for some $\pi_1,\pi_2\in S_3$, $g_1,g_2\in G$ so
    $N_G(A)\leq 2N_G(S_3)+2=6$.
    Otherwise, all other matrices from Eq.~\eqref{eq:remaining_matrices} are in $G$ (up to permutations); hence
    $N_G(B_m\underline{\pi_m})\leq 2N_G(S_3)+1=5$ and thus $N_G(A)\leq 3m+5\leq 3\cdot 5+5=20$.
\end{itemize}
Altogether this shows $N_G(s(3))\leq 20$, as claimed.
\end{proof}

Our strategy can be summarized as follows: If some set $G\subseteq s(n)$ features elementary matrices~\eqref{eq:buildingblocks}, a 2-level flip
$
{\bf1}\oplus\sigma_x
$ (up to permutations), and a cyclic permutation,
then $G$ becomes a generating set as soon as one adds all $A\in s(n) $ (up to $S_n\times S_n$) which Lemma~\ref{lemma_cond_div} does not apply to.
One fact used here is that a 2-level flip and a cyclic permutation make for the smallest generating set of $S_n$, where the number of needed factors is known to scale as $\Theta(n^2)$ 
 \cite{PermutationGroupDiameter}.
We will comment on
the problems with this proof strategy for $n\geq 4$
in the following Outlook section.

\section{Outlook}\label{sec_outlook}

Finally, we want to touch upon some open questions.
First, is the upper bound $N_G(S)\leq 20$ in Theorem~\ref{thm_s3} optimal? We conjecture that this is true, seeing how this number corresponds to a factorization $A=\underline{\pi_0}\prod_{i=1}^6g_i\underline{\pi_i}$, $\pi_i\in S_n$, $g_i\in\eqref{eq:buildingblocks}$ alternating between permutations and two-level transformations where, as discussed in the proof of our main result, generically one has to apply Lemma~\ref{lemma_cond_div} six times (in three dimensions).
However, proving lower bounds for $N_G(S)$---as seen already in Proposition~\ref{prop:G_of_s2}---becomes increasingly complicated the larger the bound is. On the other hand, one may argue that upper bounding $N_G(S)$ (as we did) is the more important task as it allows for basic error analysis, recall also Remark~\ref{rem_div_error_bound}.

Next is the question raised at the end of Section~\ref{subsec_main_2}, namely: Can the proof strategy of Theorem~\ref{thm_s3} be generalized to higher dimensions?
The reason that for $n=3$ this strategy yields a well-structured generating set $G$ is that in three dimensions, Lemma~\ref{lemma_cond_div} fails if and only if the matrix is prime.
For higher dimensions this is not true anymore---recall Eq.~\eqref{eq:indiv_s4}~ff.---meaning the structure of the generating set obtained this way is unclear for $n\geq 4$.
The problem here is that in dimensions $n\geq 4$ simply checking $\sgn(A)$ is not conclusive anymore with regards to whether $A\in s(n)$ is (in)divisible. 
As such, one would have to find other strategies and ways to decompose higher-dimensional stochastic matrices, beyond just two-level transformations.
On the other hand, one may ask how the error bound scales with increasing $n$. We conjecture that, for a generating set $G$ of a ``similar form'' an upper bound for $N_G(s(n))$ of order $\mathcal O(n^4)$ should be possible:
The permutations can be generated in $\Theta(n^2)$ (as noted earlier) and the division algorithm from Lemma~\ref{lemma_cond_div} returns something trivial after at most $n^2-n$ (maximal number of zeros in a stochastic matrix) times before becoming trivial.
However, of course, the details crucially depend on the precise generating set $G$ and whether the \textit{divisible} matrices for which Lemma~\ref{lemma_cond_div} fails are included in $G$ or whether they get broken up in another way.

Finally, it would be interesting to investigate what portion of stochastic matrices (w.r.t.~the standard Lebesgue measure) decompose into elementary matrices~\eqref{eq:buildingblocks} and permutation matrices, and what percentage inevitably ``gets stuck'' in a non-trivial indivisible element of $s(3)$ when decomposing it via Lemma~\ref{lemma_cond_div}.
This might be insightful for a fixed dimension---to understand ``how many'' matrices are of this simple form---as well as for varying dimension to understand how this portion scales with $n$,
where the latter could give insights into the ``frequency'' of such primes.

\backmatter





\bmhead{Acknowledgements}

The hospitality of Karol {\.Z}yczkowski and the Jagiellonian University, where this project was initiated, is gratefully acknowledged.
FvE has been supported by the Einstein Foundation (Einstein Research Unit on Quantum Devices) and the MATH+ Cluster of Excellence.
FS acknowledges the projects DESCOM VEGA 2/0183/2, DeQHOST APVV-22-0570, and the SAIA scholarship.

\section*{Declarations}

The authors declare that they have no competing interests.

 \bibliography{control21vJan20}

\end{document}